\documentclass[11pt,reqno]{amsart}

\usepackage{amsfonts}
\usepackage{graphicx}
\usepackage{a4wide}
\usepackage{color}
\usepackage[active]{srcltx}
\usepackage{amssymb, amsmath, mathtools}
\usepackage{float}
\usepackage[ansinew]{inputenc}
\usepackage{cancel}

\usepackage[colorlinks=true, linkcolor=blue, urlcolor=red, citecolor=blue]{hyperref}%ativa a criação de links% %tem de se retirar a \usepackage{hyperref}%.

\newcounter{main}

\newtheorem{theorem}{Theorem}[section]
\newtheorem{proposition}[theorem]{Proposition}
\newtheorem{lemma}[theorem]{Lemma}

\newtheorem{remark}{Remark}[section]

\newtheorem{maintheorem}{Theorem}
\newtheorem{maincorollary}{Corollary}

\theoremstyle{remark}

\def\NN{\mathbb{N}}

\def\RR{\mathbb{R}}
\def\UU{\mathcal{U}}

\def\plim{{p}\hbox{-}\lim}
\def\splim{{S(p)}\hbox{-}\lim}

\def\ulim{{\mathcal{U}}\hbox{-}\lim}

\numberwithin{equation}{section}

\DeclareMathOperator{\sh}{sh}

\title[Ultralimits of Birkhoff averages]{Ultralimits of Birkhoff averages}

\author[M. Carvalho]{Maria Carvalho}
\address{Maria Carvalho\\ Centro de Matem\'{a}tica da Universidade do Porto\\ Rua do
Campo Alegre 687\\ 4169-007 Porto\\ Portugal}
\email{mpcarval@fc.up.pt}

\author[F. J. Moreira]{Fernando Jorge Moreira}
\address{Fernando Jorge Moreira\\ Centro de Matem\'{a}tica da Universidade do Porto\\ Rua do Campo Alegre 687\\ 4169-007 Porto\\ Portugal}
\email{fsmoreir@fc.up.pt}

\date{\today}

\thanks{MC and FM were partially supported by CMUP (UID/MAT/00144/2013), which is funded by FCT (Portugal) with national (MEC) and European structural funds through the programs FEDER, under the partnership agreement PT2020.}


\begin{document}

\maketitle

\begin{abstract}
Given a compact metric space $X$ and a probability measure in the $\sigma-$algebra of Borel subsets of $X$, we will establish a dominated convergence theorem for ultralimits of sequences of integrable maps and apply it to deduce a non-standard ergodic-like theorem for any probability measure. %and dynamical systems acting in $X$.
\end{abstract}

\section{Introduction}

Let $(X,\mathcal{B}, \mu)$ be a measure space, where $\mathcal{B}$ is a $\sigma-$algebra of subsets of $X$ and $\mu$ is a $\sigma-$finite measure on $\mathcal{B}$, and consider a measurable map $T:X \to X$. One says that $T:X\to X$ preserves $\mu$ (or that $\mu$ is $T-$invariant) if $\mu(T^{-1}(A)) = \mu(A)$ for any $A \in \mathcal{B}$. The measure $\mu$ is said to be ergodic with respect to $T$ if, given $A\in \mathcal{B}$ with $T^{-1}(A)=A$, we have $\mu(A)\times \mu(X\setminus A)=0$. For each $x \in X$, its orbit by $T$ is defined by the sequence of iterates $\left(T^n(x)\right)_{n \,\in \, \mathbb{N}}$, where $T^0 = Id$. If $T$ preserves $\mu$ and $\mu$ is finite, the Recurrence Theorem of Poincar\'e asserts that, for every $A \in \mathcal{B}$ with $\mu(A) > 0$, the orbit of almost every point in $A$ returns to $A$ infinitely many times. If, additionally, $\mu$ is ergodic with respect to $T$, then the expected time for the first return, as estimated by Kac (cf. \cite{Petersen}), is of the order of $\frac{1}{\mu(A)}$. Besides, by the Ergodic Theorem of Birkhoff \cite{Walters, Katznelson-Weiss} we may also evaluate the mean sojourn of almost every orbit in $A$, and it is asymptotically close to $\mu(A)$. The statement of the Ergodic Theorem is in fact more general, asserting that, if $T$ preserves $\mu$ and $\varphi:X \to \mathbb{R}$ belongs to $L^1(X)$ (as happens with the characteristic map $\chi_A$ for every set $A \in \mathcal{B}$ whenever $\mu$ is a finite measure), then there exists $\widetilde{\varphi}\in L^1(X)$ such that, at $\mu-\text{a.e.} \,x \in X$, we have $\lim_{n \to +\infty}\, \frac 1 n \,\sum_{k=1}^n\,\varphi \circ T^k(x) = \widetilde{\varphi}(x)$ and $\widetilde{\varphi}\circ T(x) = \widetilde{\varphi}(x)$, and $\int_X \, \widetilde{\varphi}\,d\mu = \int_X \,\varphi d\mu.$

Several generalizations of this theorem are known, either demanding less from the obser\-vable $\varphi$ or from the probability measure $\mu$ (cf. \cite{Krengel, CM2014, CV-M-Marinacci}). The aim of this work is to settle an abstract frame for these generalizations through a non-standard dominated convergence theorem whenever $X$ is a compact metric space and $\mu$ is a Borel probability measure. Its application to Birkhoff averages of measurable bounded potentials, with res\-pect to either dynamical systems without invariant measures (such as $T:[0,1] \to [0,1]$ given by $T(x)=\frac{x}{2}$ if $x \neq 0$, $T(0)=1$) or to those whose invariant measures have relevant sets of points with historical behavior (as the ones described in \cite{Takens2, Kiriki-Soma}), conveys more information on the accumulation points of such averages. The main tools in this non-standard approach are the notions of ultrafilter and ultralimit, besides the ultrapower construction in order to produce extensions of relevant structures and transformations to the non-standard realm. (Concerning non-standard analysis, we refer the reader to \cite{Cutland, Goldblatt}.)

The paper is organized as follows. After recalling a few basic properties of ultralimits, ultraproducts, the shadow map and integrability, we will prove a Dominated Convergence Theorem for ultralimits, from which we will deduce an ergodic-like theorem for a Borel probability measure in a compact Hausdorff space, where a dynamical system $T$ is acting, and a measurable bounded function. Using the shift map in the space of ultrafilters, we will also show the existence of a space mean of the Birkhoff limits when we take into account all the possible choices of the ultrafilter.

\section{Basic definitions}

In this section we will give a brief though comprehensive list of the non-standard concepts and results we will use in the sequel. More information may be found in \cite{Goldblatt}.

\subsection{Filters}

A filter on a set $X$ is a non-empty family $\mathcal{F}$ of subsets of $X$ such that:
\begin{itemize}
\item[(i)] $A, B \in \mathcal{F} \quad \Rightarrow \quad  A \cap B \in \mathcal{F}$.
\item[(ii)] $A \in \mathcal{F} \text{ and } A \subseteq B \subseteq X \quad \Rightarrow \quad B \in \mathcal{F}$.
\item[(iii)] $\emptyset \notin \mathcal{F}$.
\end{itemize}

\medskip

A filter $\mathcal{U}$ on $X$ is said to be an \emph{ultrafilter} if for every $A \subseteq X$ either $A \in \mathcal{U}$ or $X\setminus A \in \mathcal{U}$ (but not both due to conditions (i) and (iii)). Ultrafilters are maximal filters with respect to the inclusion, and provide a useful criterion to establish which sets are considered large. Given $m \in \mathbb{N}$, the family $\mathcal{U}_{m}=\{A \subseteq\mathbb{N} \colon m \in A\}$ is an ultrafilter in  $\mathbb{N}$, called \emph{principal}. We are interested in non-principal ultrafilters as a measure of largeness of sets. For instance, take the Fr\'echet filter $\mathcal{F}_{\text{cf}} = \{A \subseteq \mathbb{N}\colon \mathbb{N}\setminus A \text{ is finite}\}$, that is, the collection of subsets of $\mathbb{N}$ whose complement is finite. There exists an ultrafilter $\mathcal{U}_{\text{cf}}$ containing $\mathcal{F}_{\text{cf}}$: since the union of any chain of proper filters is again a proper filter, by Zorn's Lemma the filter $\mathcal{F}_{\text{cf}}$ is contained in a maximal proper filter $\mathcal{U}_{\text{cf}}$. It is not hard to see that $\mathcal{U}_{\text{cf}}$ is a non-principal ultrafilter. One advantage of using this kind of ultrafilters is the fact that an ultrafilter is non-principal if and only if it contains the Fr\'echet filter of co-finite subsets.

Assuming that $\mathbb{N}$ has the discrete topology, we will denote by $\beta \mathbb{N}$ the Stone-$\check{C}$ech compactification of $\mathbb{N}$, which is a non-metrizable Hausdorff compact space (cf. \cite{Willard}). The space $\beta \mathbb{N}$ is homeomorphic to the collection $\mathbb{U}_\mathbb{N}$ of all the ultrafilters of subsets of $\mathbb{N}$ endowed with the topology generated by the (open and closed) sets $\left\{\mathbb{O}_F\right\}_{F \subseteq \mathbb{N}}$ where
$$\mathbb{O}_F=\{\mathcal{U} \in \mathbb{U}_\mathbb{N} \colon F \in \mathcal{U}\}.$$

\subsection{Ultralimits}\label{sse:ultralimits}

Let $X$ be a compact Hausdorff space and $\mathcal{U}$ be a non-principal ultrafilter on $\NN$. We say that a sequence $(x_n)_{n \in \mathbb{N}}$ in $X$ is $\mathcal{U}-$convergent in $X$ to $\ell$, and denote its limit by $\mathcal{U} \text{-} \lim_n \,x_n$, if, for any neighborhood $V_\ell$ of $\ell$, we have $\{n\in \NN  \colon x_n \in V_\ell\} \in \mathcal{U}$.  %For instance, given $n_0 \in \mathbb{N}$, we have $\mathcal{U}_{n_0} \text{-} \lim_n \, x_n = x_{n_0}$. On the other hand, the $\mathcal{F}_{\text{cf}} \text{ -} \lim_n$ coincides with the usual notion of convergence of sequences.
Observe that the $\mathcal{U} \text{-} \lim_n \,x_n$ always exists since $X$ is compact. Otherwise, for every $\ell \in X$ we may find an open neighborhood $V_\ell$ such that the set $\mathcal{C}_\ell = \{n \in \mathbb{N} \colon x_n \in V_\ell\}$ does not belong to $\mathcal{U}$. As $\mathcal{U}$ is an ultrafilter, the complement $\mathbb{N} \setminus \mathcal{C}_{\ell}$ must be in $\mathcal{U}$. Besides, as $X$ is compact, we may take a finite subcover $\left\{V_{\ell_1}, \ldots, V_{\ell_k}\right\}$ of the cover $\left(V_\ell\right)_{\ell \in X}$. The finite intersection $\bigcap_{j=1}^k\,\Big(\mathbb{N} \setminus \mathcal{C}_{\ell_j}\Big)$ is in $\mathcal{U}$ as well. Hence, as the empty set is not in $\mathcal{U}$, we conclude that $\bigcap_{j=1}^k\,\Big(\mathbb{N} \setminus \mathcal{C}_{\ell_j}\Big) \neq \emptyset$. However, if $m \in \bigcap_{j=1}^k\,\Big(\mathbb{N} \setminus \mathcal{C}_{\ell_j}\Big)$, then $x_m$ does not belong to $V_\ell$ for every $1 \leq \ell \leq k$. This contradicts the fact that $\left\{V_{\ell_1}, \ldots, V_{\ell_k}\right\}$ is a cover of $X$. Additionally, as $X$ is Hausdorff, the ultralimit is unique.

Given a sequence $(x_n)_{n \in \mathbb{N}}$, the ultralimits for all the possible choices of non-principal ultrafilters in $\mathbb{N}$ are precisely the cluster points of this sequence. In particular, if the sequence is convergent in $X$ to a limit $\ell$, then $\mathcal{U} \text{-} \lim_n \, x_n = \ell$ for every non-principal ultrafilter $\mathcal{U}$ in $\mathbb{N}$. If a real-valued sequence is not bounded, its $\UU-$limit always exists, though it may be either $+\infty$ or $-\infty$.

\subsection{Ultrapower construction}

Let $X$ be a compact metric space, $\mu$ be a probability measure defined on the $\sigma-$algebra $\mathcal{B}$ of the Borel subsets $X$ and $\mathcal{U}$ be an ultrafilter in $\mathbb{N}$. Given two sequences $\left(a_n\right)_{n \,\in\, \mathbb{N}}$ and $\left(b_n\right)_{n \,\in\, \mathbb{N}}$ of elements of $X$, define the equivalence relation
$$\left(a_n\right)_{n \,\in\, \mathbb{N}} \quad \backsim \quad \left(b_n\right)_{n \,\in\, \mathbb{N}} \quad \quad \Leftrightarrow \quad \quad \Big\{n \in \NN \colon \,a_n \,= \, b_n\Big\} \,\, \in \mathcal{U}.$$
Denote by $\widehat X$ the \emph{ultrapower} of $X$ made by the equivalence classes of sequences of elements of $X$, that is,
$$\widehat X \,\,:= \,\,X^\NN_{\diagup \backsim} \,\,= \,\,\Big\{[\left(x_n\right)_{n \,\in \,\NN}] \colon x_n \in X \,\,\, \forall n \in \NN\Big\}.$$
In what follows we will denote by $[x_n]$ the equivalence class of the sequence $\left(x_n\right)_{n \, \in \, \NN}$ in $\widehat X$. In a natural way, $X$ is embedded in $\widehat X$ by the inclusion map $\iota: \,X \hookrightarrow \widehat X$ given by $\iota(x)=[x]$, where $[x]$ stands for the equivalence class of the constant sequence whose term is constant and equal to $x$.

\subsection{The shadow map}

Part of the usefulness of the ultrapower structures relies on the possibility to transfer information from its universe to the standard realm. To do it, one often uses the \emph{shadow map} $\sh_{_{\mathcal{U},\,X}}\colon \widehat X \to X$ defined by
$$\sh_{_{\mathcal{U},\,X}}([\left(a_n\right)_{n \, \in \, \NN}])\,\, := \,\,\mathcal{U} \text{-} \lim_n \, a_n.$$

\begin{lemma}\label{le:openclosedsets} $\,$
\begin{enumerate}
\item $C\subset X$ is closed if and only if $ \widehat C \subset \sh_{_{\mathcal{U},\,X}}^{-1}(C)$.
\item $O\subset X$ is open if and only if  $\sh_{_{\mathcal{U},\,X}}^{-1}(O)\subset \widehat O$.
\end{enumerate}
\end{lemma}

\begin{proof}
Let $C$ be a closed subset of $X$ and suppose that there exists $\widehat{x}:=[\left(x_n\right)_{n \, \in \, \NN}] \in \widehat C$ such that $y=\sh_{_{\mathcal{U},\,X}}(\widehat{x}) \not\in C$. Since $C$ is closed we can find $\epsilon >0$ such that the open ball $B_\epsilon(y)$ of radius $\epsilon$ centered at $y$ is contained in $X \setminus C$. By definition $y=\ulim_n x_n$, and thus $\left\{n \in \NN\; :\ x_n \in B_\epsilon(y) \right\} \in \mathcal{U}$. Therefore, $\left\{n \in \NN\; :\ x_n \in X \setminus C \right\} \in \mathcal{U}$, or equivalently, $\widehat{x} \in \widehat{X \setminus C}$. However, $\widehat{X \setminus C} = \widehat{X} \setminus \widehat{C}$, so $\widehat{x} \notin \widehat C$, which is a contradiction. Assume now that $\widehat C \subset \sh_{_{\mathcal{U},\,X}}^{-1}(C)$. If $C$ were not closed, then one could find $y \in X \setminus C$ such that, for every $n \in \NN$, we would have $B_{1/n}(y) \cap C \ \not= \ \emptyset.$ For each $n \in \NN$, take $c_n \in B_{1/n}(y) \,\cap\, C$. Since $y=\lim_n c_n$, then $y=\sh_{_{\mathcal{U},\,X}}([\left(c_n\right)_{n \, \in \, \NN}])$, so $\widehat{c}:=[\left(c_n\right)_{n \, \in \, \NN}] \in \sh_{_{\mathcal{U},\,X}}^{-1}(X \setminus C)$. Now, $\sh_{_{\mathcal{U},\,X}}^{-1}(X \setminus C)= \widehat{X}\setminus \sh_{_{\mathcal{U},\,X}}^{-1}(C)$, thus $\widehat{c} \notin \sh_{_{\mathcal{U},\,X}}^{-1}(C)$. As $\widehat C \subset \sh_{_{\mathcal{U},\,X}}^{-1}(C)$, we have $\widehat{X} \setminus \widehat C \supset \widehat{X} \setminus \sh_{_{\mathcal{U},\,X}}^{-1}(C)$, hence $\widehat{c} \not \in \widehat{C}$. Yet, by construction, $\widehat{c}=[\left(c_n\right)_{n \, \in \, \NN}]$ belongs to $\widehat{C}$. Concerning item (2), we are left to notice that $O \subset X$  is open if and only if $X \setminus O$ is closed, to recall that $\widehat{X \setminus O} \subset \sh_{_{\mathcal{U},\,X}}^{-1}(X \setminus O)$ by the previous item, and that
$$\widehat{X \setminus O} \subset \sh_{_{\mathcal{U},\,X}}^{-1}(X \setminus O) \quad \quad \Longleftrightarrow \quad \quad \sh_{_{\mathcal{U},\,X}}^{-1}(O) \subset \widehat{O}.$$
\end{proof}

\subsection{Ultraproducts and a finitely additive measure}

A set $\Lambda \subset \widehat X$ is said to be an \emph{ultraproduct set} (which we abbreviate to \emph{UP-set}) if there exists a sequence $\left(A_n\right)_{n \,\in\, \mathbb{N}}$ of subsets of $X$ such that
$$\Lambda = \Big(\prod_{n \,\in \,\NN} \,A_n\Big)_{\diagup \backsim}.$$
For UP-sets that are determined by sequences of Borel sets, which form an algebra we denote by $\widehat{\mathcal{B}}$, we have defined a finitely additive measure as follows: if $\Lambda = \Big(\prod_{n \,\in \,\NN} \,A_n\Big)_{\diagup \backsim}$ where $A_n \in \mathcal{B}$ for every $n \in \mathbb{N}$, then
\begin{equation}\label{def:medida Loeb}
\widehat \mu_{_\mathcal{U}}(\Lambda) \,:= \,\,\mathcal{U}\text{-}\lim_n \,\mu(A_n).
\end{equation}
The previous computation does not depend on the sequence $\left(A_n\right)_{n \,\in\, \mathbb{N}}$ of measurable subsets of $X$ whose product builds $\Lambda$ (cf. \cite{Goldblatt}).

\subsection{$\UU-$integrability}\label{sse:U-integrable}
Let $X$ be a compact metric space, $\mu$ be a probability measure defined on the $\sigma-$algebra $\mathcal{B}$ of the Borel subsets $X$ and $\mathcal{U}$ be a non-principal ultrafilter in $\mathbb{N}$. A sequence $(f_n)_{n \, \in \, \mathbb{N}}$ of $\mathcal{B}-$measurable functions $f_n: X \to \RR$ is said to be $\UU-$integrable if the following conditions hold:
\begin{enumerate}
\item For every $n \in \NN$, the map $f_n$ is $\mu-$integrable.
\medskip
\item $\mathcal{U} \text{-} \lim_n \,\int_X\,|f_n|\,d\mu$ is finite.
\medskip
\item If $\left(B_n\right)_{n \, \in \, \NN}$ is a sequence of elements of $\mathcal{B}$, then
$$\mathcal{U} \text{-} \lim_n \,\mu(B_n)=0 \quad \quad \Rightarrow \quad \quad \mathcal{U} \text{-} \lim_n \, \int_{B_n}\,|f_n|\,d\mu  =  0.$$
\end{enumerate}

For example, take $B \in \mathcal{B}$, the characteristic map $\chi_{_B}$ of $B$ and, for each natural number $n$, consider $f_n=\chi_{_B} \circ T^n$. Then
$$\mathcal{U} \text{-} \lim_n \,\int_X\,f_n\,d\mu = \mathcal{U} \text{-} \lim_n \, \mu(T^{-n}(B)) \in [0,1]$$
and, for every sequence $\left(B_n\right)_{n \, \in \, \NN}$ of elements of $\mathcal{B}$ satisfying $\mathcal{U} \text{-} \lim_n \,\mu(B_n)=0$, we have
$$0 \leq \mathcal{U} \text{-} \lim_n \, \int_{B_n}\,f_n\,d\mu  =  \mathcal{U} \text{-} \lim_n \, \mu(T^{-n}(B_n) \cap B_n) \leq \mathcal{U} \text{-} \lim_n \, \mu(B_n) = 0.$$

By a similar argument, we conclude that if $\varphi: X \to \RR$ is a measurable and bounded map (so, for every $n \in \NN_0$, the map $\varphi \circ T^n$ is $\mu-$integrable), then the sequence $(f_n)_{n \, \in \, \mathbb{N}}$ defined by
$$f_n := \frac{1}{n}\Big(\varphi + \varphi \circ T + \cdots + \varphi \circ T^{n-1}\Big)$$
is $\UU-$integrable.

On the contrary, the unbounded sequence of maps $f_n:[0,1] \to \RR$ defined by $f_n(x)= n$ if $x < 1/n$ and $f_n(x)=0$ otherwise is not $\UU-$integrable if we consider in $[0,1]$ the Lebesgue measure $m$. Indeed, for every $n \in \NN$, the map $f_n$ is Lebesgue-integrable and $\lim_n \,\int |f_n|\,d\,m = 1$. However, if $B_n=[0, \frac{1}{n}]$ for every $n \in \NN$, then $\lim_n \,m(B_n)=0$ but $\int_{B_n}\,|f_n|\,d\,m = 1.$ We observe that, in this case, 
$\int \,\lim_n f_n \,d\,m \neq \lim_n \int f_n\,d\,m$.

\section{Main results}

After selecting a non-principal ultrafilter in $\mathbb{N}$ and extending the finitely additive measure \eqref{def:medida Loeb} to a measure defined on a $\sigma-$algebra on the ultrapower $\widehat{X}$ that contains $\sh_{_{\mathcal{U},\,X}}^{-1}(\mathcal{B})$, we will consider integrable maps and the ultralimits of their Birkhoff averages in order to prove a non-standard pointwise convergence theorem for any probability measure.

There are two main difficulties to establish such a result. The first one is the lack of some version of the Dominated Convergence Theorem with respect to ultralimits. We will establish the following one.

\begin{maintheorem}\label{teo:main} If $g: X \to \RR$ is $\mu-$integrable and $(f_n)_{n \, \in \, \mathbb{N}}$ is a sequence of real-valued $\mu-$integrable functions with $|f_n| \leq g$ for every $n \in \NN$, then the map
$$[x_n] \,\in \,\widehat{X} \quad \mapsto \quad \mathcal{U}\text{-}\lim_n \,[f_n(x_n)]$$
is $\widehat \mu_{{_\mathcal{U}}}-$integrable and satisfies
$$\int_{\widehat X} \,\mathcal{U}\text{-}\lim_n \,[f_n(x_n)]\,d\,\widehat \mu_{{_\mathcal{U}}}([x_n]) \,\,=\, \,\mathcal{U}\text{-}\lim_n \, \int_X \,f_n(x) \,\,d\mu(x).$$
\end{maintheorem}

\medskip

This result is a consequence of a more general statement concerning bounded $\UU-$integrable sequences $(f_n)_{n \, \in \, \mathbb{N}}$ of real-valued functions which we will prove in Section~\ref{se:prova-teorema-A}.

\medskip

\begin{maintheorem}\label{teo:second-main} If $(f_n)_{n \, \in \, \mathbb{N}}$ is a $\UU-$integrable sequence of real-valued functions, then the map
$$[x_n] \,\in \,\widehat{X} \quad \mapsto \quad \mathcal{U}\text{-}\lim_n \,[f_n(x_n)]$$
is $\widehat \mu_{{_\mathcal{U}}}-$integrable and satisfies
$$\int_{\widehat X} \,\mathcal{U}\text{-}\lim_n \,[f_n(x_n)]\,d\,\widehat \mu_{{_\mathcal{U}}}([x_n]) \,\,=\, \,\mathcal{U}\text{-}\lim_n \, \int_X \,f_n(x) \,\,d\mu(x).$$
\end{maintheorem}

\medskip

The second difficulty concerns the measurability of the ultralimit: although the pointwise convergence with respect to an ultrafilter of a sequence of measurable functions is guaranteed, its ultralimit may not be measurable. Let us see an example. Consider a compact metric space $X$, a $\sigma-$algebra of subsets of $X$, an ultrafilter $\mathcal{U}$ and a sequence $(f_n)_{n \, \in \, \mathbb{N}}$ of measurable functions $f_n:X \to \mathbb{R}$. If $\mathcal{U}$ is principal, generated by $\{n_0\}$, then the $\mathcal{U} \text{-} \lim_n \,f_n$ is $f_{n_0}$, so it is measurable. Otherwise, if $\mathcal{U}$ is non-principal, consider $X=[0,1]$ with the usual topology and the Lebesgue measure. For each $n \in \mathbb{N}$ and $x \in X$, define
$$f_n(x) = \text{the $n$th digit in the infinite binary expansion of $x$}.$$
Then the $\mathcal{U} \text{-} \lim_n \,f_n$ sends $x \in X$ to $1$ if and only if the set
$$\{n \in \mathbb{N} \colon \text{the $n$th bit in the infinite binary expansion of $x$ is 1}\}$$
is in $\mathcal{U}$. In other words, if we identify $x$
via its binary expansion with a sequence of $0$'s and $1$'s and if we then regard that sequence as the characteristic function of a subset of $\mathbb{N}$, then the $\mathcal{U} \text{-} \lim_n \,f_n$, mapping subsets of $\mathbb{N}$ to $\{0,1\}$, is just the characteristic function of $\mathcal{U}$. However, a theorem of Sierpinski \cite{Sierpinski} \footnote{We thank Andreas Blass for calling our attention to this reference.}
asserts that this is never Lebesgue measurable when $\mathcal{U}$ is a non-principal ultrafilter. To overcome this problem, we summon the ultrapower extension of the space $X$ with respect to a fixed ultrafilter. This way, we are able to prove the following property of the Birkhoff averages of bounded measurable potentials $\varphi$ as a direct consequence of Theorem~\ref{teo:main} when applied to the sequence
\begin{eqnarray}\label{eq:averages}
(f_n)_{n \, \in \, \NN} = \Big(\frac 1 n \,\sum_{j=0}^{n-1}\,\varphi \circ T^j\Big)_{n \, \in \, \NN}
\end{eqnarray}
whenever $(f_n)_{n \, \in \, \NN}$ is $\UU-$integrable.

\begin{maincorollary}\label{cor:main}
Let $X$ be a compact metric space, $\mu$ be a probability measure defined on the Borel subsets of $X$ and $\mathcal{U}$ be a non-principal ultrafilter in $\NN$. Consider a measurable map $T: X \to X$ and a measurable bounded function $\varphi: X \to \RR$. Then there exists a $\widehat \mu_{_\mathcal{U}}-$integrable map $\widehat \varphi_{_\mathcal{U}}: \widehat X \to \RR$ satisfying:

\begin{enumerate}
\item $\mathcal{U} \text{-} \lim_n\,\frac 1 n \,\sum_{j=0}^{n-1}\,\varphi \circ T^j(x)  =  \widehat \varphi_{_\mathcal{U}} \circ \iota(x)$ for every $x \in X$. %for $\mu-$almost every $x \in X$.
\bigskip
\item $\widehat \varphi_{_\mathcal{U}}\circ \iota(T(x)) = \widehat \varphi_{_\mathcal{U}}\circ \iota(x)$ for every $x \in X$. %for $\mu-$almost every $x \in X$.
\bigskip
\item $\ulim_n \,\frac 1 n \,\sum_{j=0}^{n-1}\, \int_{_{X}} \,\varphi\circ T^j \, d\mu = \int_{_{\widehat X}}\,\widehat \varphi_{_\mathcal{U}}\,d\,\widehat \mu_{_\mathcal{U}}.$
\end{enumerate}
\end{maincorollary}

\medskip

As an immediate consequence of Corollary \ref{cor:main} and \cite{Katznelson-Weiss} we deduce that, if $\mu$ is $T-$invariant, then, given a measurable bounded function $\varphi: X \to \RR$ and an ultrafilter $\mathcal{U}$, we have:
\begin{itemize}
\item[(a)] The maps $\widehat \varphi_{_\mathcal{U}}\circ \iota$ and $\widetilde{\varphi}$ (given by the Ergodic Theorem applied to $T$ and $\varphi$) coincide $\mu$ almost everywhere.
\medskip
\item[(b)] $\int_{_{\widehat X}}\,\widehat \varphi_{_\mathcal{U}}\,d\,\widehat \mu_{_\mathcal{U}} = \int_{_{X}} \,\varphi \, d\mu.$
\medskip
\item[(c)] $\int_X \,\mathcal{U} \text{-} \lim_n\,\,\,\frac 1 n \,\sum_{j=0}^{n-1}\,\varphi \circ T^j(x)\,\,d\mu(x) = \int_X \,\varphi(x) \,\,d\mu(x).$
\end{itemize}

\medskip

A natural question we may now address concerns the impact of the choice of the ultrafilter $\UU$. %Is there an average Birkhoff limit in the space of all non-principal ultrafilters?
Using the shift map on the space $\beta \NN$ of all non-principle ultrafilters in $\NN$ and a suitable Borel shift-invariant probability measure on $\beta \NN$ of all non-principle ultrafilters, we will show in Section~\ref{se:proof-of-Corollary-B} that there exists in $\beta \NN$ a space mean  of all the ultralimits of the Birkhoff averages \eqref{eq:averages}.

\section{Construction of an ultrapower measure}

In this section we recall how to extend $\widehat \mu_{_\mathcal{U}}$ to a $\sigma-$algebra containing $\widehat{\mathcal{B}}$.

\begin{lemma}\cite[Theorem 11.10.1]{Goldblatt}\label{le:fip}
Let $(\Lambda_n)_{n\, \in \,\mathbb{N}}$ be a decreasing sequence of non-empty UP-sets $\Lambda_n$. Then $\bigcap_{n \,\in \,\NN}\,\,\Lambda_n \not= \emptyset.$ Therefore, any cover of a UP-set by countably many UP-sets has a finite subcover.
\end{lemma}

Since $\widehat \mu_{_\mathcal{U}}$ is finitely additive and for any disjoint union $\Lambda \ = \ \dot{\bigcup}_{k \, \in \, \mathbb{N}} \,\Lambda_k$ of UP-sets $\Lambda_k$ there is $k_0 \in \NN$ such that $\Lambda  \ = \  \Lambda_1\dot{\cup} \ldots \dot{\cup}\,\Lambda_{k_0}$, with $\Lambda_k =\emptyset$ for every $k > k_0$, we conclude that the compatibility condition
\begin{equation}\label{caratheodor_cond}
\Lambda = \dot{\bigcup}_{k \, \in \, \mathbb{N}} \,\Lambda_k \quad \quad \Rightarrow \quad \quad \widehat \mu_{_\mathcal{U}}(\Lambda) = \sum_{k=1}^\infty\,\,\widehat \mu_{_\mathcal{U}}(\Lambda_k).
\end{equation}
is valid. Consequently,

\begin{proposition}\cite{AB1998}\label{LBcomplete}
The finitely additive measure $\widehat \mu_{_\mathcal{U}}$ can be extended to a measure, we will keep denoting by $\widehat {\mu}_{_\mathcal{U}}$, on a $\sigma-$algebra $L(\widehat{\mathcal{B}})$ which satisfies:
\begin{enumerate}
\item $L(\widehat{\mathcal{B}}) \supset \widehat{\mathcal{B}}$.
\item $L(\widehat{\mathcal{B}})$ is a complete $\sigma-$algebra, that is, given $\Gamma \subset \widehat X$, if there are
  $\Gamma_1, \, \Gamma_2 \in L(\widehat{\mathcal{B}})$ such that $\Gamma_1 \subset \Gamma \subset \Gamma_2$ and $\widehat {\mu}_{_\mathcal{U}}(\Gamma_1)=\widehat {\mu}_{_\mathcal{U}}(\Gamma_2)$, then $\Gamma \in L(\widehat{\mathcal{B}}).$
\end{enumerate}
\end{proposition}

Under the assumption that $X$ is a metric space and $\mu$ a Borel probability measure, we know that $\mu$ is a regular measure (cf. \cite[Corollary 6.1.1]{Walters}). More precisely, for any Borel set $B \in \mathcal{B}$ there is a decreasing sequence $(O_n)_{n \, \in \,\NN}$ of open sets and an increasing sequence $(C_n)_{n \, \in \,\NN}$ of closed sets such that
\begin{equation}\label{caractborelian}
C_n \subset B  \subset O_n \quad \quad \text{and} \quad \quad \mu(O_n)-\mu(C_n) < \frac 1n \quad \quad \forall \,\,n \in \NN.
\end{equation}
From this property we deduce that the $\sigma-$algebra $\sh_{_{\mathcal{U},\,X}}^{-1}(\mathcal{B})$ is contained in $L(\widehat{\mathcal{B}})$ and that $\sh_{_{\mathcal{U},\,X}}$ is a measure preserving map.

\begin{lemma}\label{sh-measurability}
For every Borel set $B \in \mathcal{B}$ we have $\sh_{_{\mathcal{U},\,X}}^{-1}(B) \in L(\widehat{\mathcal{B}})$ and $\widehat {\mu}_{_\mathcal{U}}(\sh_{_{\mathcal{U},\,X}}^{-1}(B)) = \mu(B).$
\end{lemma}

\begin{proof}
For each $B \in \mathcal{B}$, consider sequences $(O_n)_{n \, \in \,\NN}$ and $(C_n)_{n \, \in \,\NN}$ as in \eqref{caractborelian}. Then
$$\sh_{_{\mathcal{U},\,X}}^{-1}(C_n) \subset \sh_{_{\mathcal{U},\,X}}^{-1}(B) \subset \sh_{_{\mathcal{U},\,X}}^{-1}(O_n) \quad \quad \forall \,n \,\in \,\mathbb{N}.$$
By Lemma~\ref{le:openclosedsets}, we have $\widehat C_n \subset \sh_{_{\mathcal{U},\,X}}^{-1}(C_n)$ and $\sh_{_{\mathcal{U},\,X}}^{-1}(O_n) \subset \widehat O_n$ for every $n \in \NN$. Thus,
$$\widehat C_n \subset \sh_{_{\mathcal{U},\,X}}^{-1}(B) \subset \widehat O_n \quad \quad \forall \,n \,\in \,\mathbb{N}.$$
Set
$$\widehat{C} := \bigcap_{n \,\in \,\NN} \,\widehat C_n \quad \quad \text{and} \quad \quad \widehat{O} := \bigcap_{n\, \in\, \NN}\, \widehat O_n.$$
Then,
$$\widehat{C} \subset \sh_{_{\mathcal{U},\,X}}^{-1}(B) \subset \widehat{O}.$$
Moreover, by Definition~\ref{def:medida Loeb} we have
$$\widehat \mu_{{_\mathcal{U}}}(\widehat C_n) = \mu(C_n) \quad \quad \text{and} \quad \quad  \widehat \mu_{{_\mathcal{U}}}(\widehat O_n) = \mu(O_n)$$
so, as $\mu(O_n)-\mu(C_n) < \frac 1n$,
$$\widehat \mu_{{_\mathcal{U}}}(\widehat{C})=\widehat \mu_{{_\mathcal{U}}}(\widehat{O}).$$
Hence, with Proposition~\ref{LBcomplete} we confirm that $\sh_{_{\mathcal{U},\,X}}^{-1}(B) \in L(\widehat{\mathcal{B}})$. Additionally, we deduce that
$$\widehat {\mu}_{{_\mathcal{U}}}(\sh_{_{\mathcal{U},\,X}}^{-1}(B)) = \widehat {\mu}_{{_\mathcal{U}}}(\widehat{C}) = \mu(B).$$
\end{proof}

We will now prove that any element of $L(\widehat{\mathcal{B}})$ differs from an element of $\widehat{\mathcal{B}}$ by a $\widehat \mu_{{_\mathcal{U}}}-$null set.

\begin{proposition}\label{LBapproximation}
For any set $\Gamma \in L(\widehat{\mathcal{B}})$ there is $\Lambda \in \widehat{\mathcal{B}}$ such that $\widehat \mu_{{_\mathcal{U}}}(\Gamma \,\Delta \,\Lambda)=0.$
\end{proposition}

\begin{proof}
According to Carathéodory Extension Theorem, we can find an increasing sequence of sets $\Sigma_n \in \widehat{\mathcal{B}}$ and a decreasing sequence of sets $\Upsilon_n \in  \widehat{\mathcal{B}}$ such that, for every $n \in \NN$,
\begin{equation}\label{defCnDn1}
\Sigma_n \subset \Gamma \subset \Upsilon_n
\end{equation}
and
\begin{equation}\label{defCnDn2}
\widehat {\mu}_{{_\mathcal{U}}}(\Sigma_n)-\widehat {\mu}_{{_\mathcal{U}}}(\Upsilon_n)  < \frac 1n.
\end{equation}
By definition, for each $n \in \mathbb{N}$ there are sequences $(A_{n,k})_{k \,\in\,\mathbb{N}}$ and $(B_{n,k})_{k \,\in\,\mathbb{N}}$ of Borel subsets of $X$ such that
$$\Sigma_n=\Big(\prod_{k \, \in \, \mathbb{N}} \,A_{n,k}\Big)_{\diagup \backsim} \quad \quad \text{and} \quad \quad \Upsilon_n=\Big(\prod_{k \, \in \, \mathbb{N}} \,B_{n,k}\Big)_{\diagup \backsim}.$$
Given $n \in \NN$, denote by $K_n$ the set of $k \in \mathbb{N}$ such that
\begin{equation}
A_{n-1,k} \subset A_{n,k} \subset B_{n,k} \subset B_{n-1,k} \quad \quad \text{and} \quad \quad \mu(B_{n,k})-\mu(A_{n,k}) < \frac 1n.
\end{equation}
We observe that $K_n \in \mathcal{U}$ for every $n \in \NN$. Since $\widehat{\mathcal{B}}$ satisfies the finite intersection property stated in Lemma~\ref{le:fip}, we may assume that the sequence $(K_n)_{n \,\in\,\NN}$ is decreasing. Moreover, as all co-finite sets belong to $\mathcal{U}$, we may also suppose that, if $k \in K_n$, then $k \geq n$. By setting $K_0=\NN$ and, for $n \in \mathbb{N}$,
$$J_n = K_{n-1} \setminus K_n$$
we obtain a sequence $(J_n)_{n \, \in \, \mathbb{N}}$ of disjoint subsets of positive integers with the property
\begin{equation}\label{defInbyJn}
K_n = \bigcup_{m\,>\,n}\,J_m.
\end{equation}
Define the UP-set
$$\Lambda := \Big(\prod_{k \, \in \, \mathbb{N}} \,\Lambda_{k}\Big)_{\diagup \backsim}$$
where $\Lambda_k := B_{m,k}$ for $k \in J_m$. We are left to prove that the set $\Lambda$ satisfies
\begin{equation}\label{eq:inclusion}
\Sigma_n  \subset \Lambda \subset \Upsilon_n \quad \quad \forall  \,\,n \in \NN
\end{equation}
because, as $\Sigma_n \subset \Gamma$, if $\Lambda \subset \Upsilon_n$ for every $n \in \mathbb{N}$ then the assertion $\widehat \mu_{{_\mathcal{U}}}(\Gamma \,\Delta\, \Lambda)=0$ is a straightforward consequence of \eqref{defCnDn2}.

To prove \eqref{eq:inclusion}, it is enough to show that, for any $n \in \NN$, we have
$$M_n = \left\{k \in \NN \colon  A_{n,k} \subset \Lambda_k \subset B_{n,k}\right\} \in \mathcal{U}.$$
Let $n \in \NN$ and $m > n$ any arbitrary integer. For $k \in J_m$ we have, by definition, $\Lambda_k = B_{m,k}$. Besides, as $B_{m,k} \subset B_{m-1,k}\subset  \cdots \subset B_{n,k}$, for $k \in J_m$ we have
$$\Lambda_k \subset B_{n,k}.$$
Therefore, $\Lambda_k \subset B_{n,k}$ for any $k \in \bigcup_{m\,>\,n} \, J_m$. So, by \eqref{defInbyJn}, we deduce that $\Lambda_k \subset B_{n,k}$ for every $k \in K_n$. As $K_n \in \mathcal{U}$, we finally obtain $\Lambda \subset \Upsilon_n$. In an analogous way one shows that $\Sigma_n \subset \Lambda$.
\end{proof}

\medskip

If we endow the space $\widehat X$ the quotient topology of the product topology in $\prod_{n \, \in \,\NN}\,X$, it will have the indiscrete topology because any nonempty open set in $\prod_{n \, \in \,\NN}\,X$ depends on only finitely many coordinates. Although with this choice $\widehat X$ would be compact, we are not interested in such a non-Hausdorff space. Otherwise, if we choose to give $\widehat X$ the so called ultraproduct topology (cf. \cite{Bankston, Bankston2}), which is generated by the sets $\Big(\prod_{n \, \in \,\mathbb{N}}\,O_n\Big)_{\diagup \backsim}$ where each $O_n$ is open in $X$, and for which the subsets $\Big(\prod_{n \, \in \,\mathbb{N}}\,F_n\Big)_{\diagup \backsim}$, where $F_n$ is closed in $X$ for every $n$, are closed, then the Borel sets are included in $L(\widehat{\mathcal{B}})$. Notice that the $\sigma-$algebra we are considering in $\widehat X$ contains the Borel sets of $\widehat X$. Moreover, by Proposition~\ref{LBapproximation}, given $\Gamma \in L(\widehat{\mathcal{B}})$ there is $\Lambda=\Big(\prod_{n \, \in \,\mathbb{N}}\,A_n\Big)_{\diagup \backsim} \in \widehat{\mathcal{B}}$ such that $\widehat \mu_{{_\mathcal{U}}}(\Gamma \,\Delta \,\Lambda)=0.$ As $\mu$ is regular, for each $A_n$ there exists a closed subset $K_n \subset A_n$ such that $\mu(A_n) - \mu(K_n) < \frac{1}{n}$. Take $\Sigma = \Big(\prod_{n \, \in \,\mathbb{N}}\,K_n\Big)_{\diagup \backsim}$; as $\ulim_n \,\mu(K_n) = \ulim_n \,\mu(A_n)$, we get
\begin{equation}\label{eq:approx-compact}
\widehat \mu_{{_\mathcal{U}}}(\Sigma \,\Delta \,\Gamma) = 0.
\end{equation}
So, although with this topology $\widehat X$ may not be locally compact (cf. \cite{Bankston}), the probability measure $\widehat \mu$ is inner regular.

\begin{remark}
\emph{We observe that the inclusion map $\iota$ may be non-measurable.\footnote{https://terrytao.wordpress.com/2008/10/14/non-measurable-sets-via-non-standard-analysis/}}
\end{remark}

\section{Proof of Theorem~\ref{teo:main}}\label{se:prova-teorema-A}

As mentioned previously, for a sequence of measurable maps $f_n: X \to \RR$ we do not necessarily have that $f=\ulim_n f_n$ is a measurable function. However, we can extend $f$ to a map $\widehat f: \widehat X \to \RR$ using the shadow map $\sh_{_{\mathcal{U},\,\RR}}: \widehat{\RR} \to \RR$ by defining
\begin{equation}\label{def:mathfrak f}
\widehat f([x_n]) \ = \ \ulim_n \,[f_n(x_n)] \ = \ \sh_{_{\mathcal{U},\,\RR}}([f_n(x_n)])
\end{equation}
(thus, if we identify $x$ with its equivalent class $[x]$, one has $\widehat f([x])= \ \ulim_n \, f_n(x)$). One advantage of performing this extension is the following.

\begin{lemma}\label{measurability_ulim_fn}
%For any sequence of $\mathcal{B}-$measurable functions $f_n: X \to \RR$,
The map $\widehat f$ is $L(\widehat{\mathcal{B}})-$measurable.
\end{lemma}

\begin{proof}
Denote by $\mathcal{B}_\RR$ the Borel sets of $\RR$. It is straightforward to verify that the map $F:\widehat X \to \widehat{\RR}$ defined by
$$F([x_n])=[f_n(x_n)]$$
is a measurable function from $(\widehat X,\,L(\widehat{\mathcal{B}}))$ to $(\widehat{\RR},\, L(\widehat{\mathcal{B}}_\RR))$. Indeed, we have $F^{-1}(\widehat{\mathcal{B}_\RR})\subset \widehat{\mathcal{B}}$ since, for each $n \in \mathbb{N}$, the map $f_n$ is measurable and so $f_n^{-1}(\mathcal{B}_\RR) \subset \mathcal{B}$. To end the proof we are left to take into account that $\sh_{_{\mathcal{U},\,\RR}}$ is measurable (cf. Lemma~\ref{sh-measurability}) and that $\widehat f= \sh_{_{\mathcal{U},\,\RR}}\,\circ\, F$.
\end{proof}

Given a $\mu-$integrable map $g: X \to \RR$ and a sequence $(f_n)_{n \, \in \, \mathbb{N}}$ of $\mu-$integrable real-valued functions satisfying $|f_n| \leq g$ for every $n \in \NN$, then $(f_n)_{n \, \in \, \mathbb{N}}$ is $\UU-$integrable. Indeed, conditions (2) and (3) of Subsection~\ref{sse:U-integrable} are immediate consequences of the domination $|f_n| \leq g$ by a $\mu-$integrable map $g$.

\begin{proposition}\label{prop:auxiliary}
If $(f_n)_{n \, \in \, \mathbb{N}}$ is a sequence of real-valued bounded $\UU-$integrable functions, then $\widehat f$ is $\widehat {\mu}_{{_\mathcal{U}}}-$integrable and
$$\int_{\widehat X} \,\widehat f \,\,d\,\widehat {\mu}_{{_\mathcal{U}}} \,\,=\, \,\mathcal{U}\text{-}\lim_n \, \int_X \,f_n \,\,d\mu.$$
\end{proposition}

\begin{proof}
As usual, we will verify the statement for sequences of characteristic functions of measurable subsets of $X$, then proceed to bounded sequences of simple functions and finally bounded sequences of measurable  functions.

\subsubsection*{\textbf{\emph{Sequences of measurable characteristic maps}}} We start considering a sequence of maps $(f_n)_{n \, \in \, \NN} = (\chi_{_{A_n}})_{n \, \in \, \NN}$, where $\chi_{_{A_n}}$ stands for the characteristic function of a measurable subset $A_n$ of $X$. As $\widehat f$ is well defined (cf. Subsection~\ref{sse:ultralimits}), measurable and bounded, it is $\widehat {\mu}_{{_\mathcal{U}}}-$integrable. Our aim is to prove that
$$\int_{\widehat X}\,\ulim_n \,[\chi_{_{A_n}}(x_n)] \,\,d\,\widehat {\mu}_{{_\mathcal{U}}}([x_n]) \ = \  \ulim_n\,\int_{X}\,\chi_{_{A_n}}(x)\,\,d\mu(x).$$
Notice that, by definition, if $\Lambda = \Big(\prod_{n\, \in\, \mathbb{N}} A_n\Big)_{\diagup \backsim}$, then
\begin{equation}\label{ulimchi_n}
\ulim_n \,[\chi_{_{A_n}}(x_n)] \ = \ \chi_{_\Lambda}([x_n]) \quad \quad  \forall \, [\left(x_n\right)_{n \, \in \, \NN}] \in \widehat X.
\end{equation}
Consequently, by \eqref{def:medida Loeb},
\begin{eqnarray}\label{eq:caracteristic}
\int_{\widehat X}\,\ulim_n \,[\chi_{_{A_n}}(x_n)] \,\,d\,\widehat {\mu}_{{_\mathcal{U}}}([x_n]) & = & \int_{\widehat X}\,\chi_{_\Lambda}([x_n])\,\,d\,\widehat {\mu}_{{_\mathcal{U}}}([x_n]) = \widehat {\mu}_{{_\mathcal{U}}}(\Lambda) \nonumber \\
&=& \ulim_n\,\mu(A_n) \nonumber \\
&=& \ulim_n\,\int_{X}\,\chi_{_{A_n}}(x)\,\,d\mu(x).
\end{eqnarray}

\subsubsection*{\textbf{\emph{Bounded sequences of simple measurable maps}}}
We proceed considering a bounded $\UU-$integrable sequence of simple functions $f_n: X \to \RR$ defined by
$$f_n=\sum_{k=1}^{p_n} a_{n,k} \,\chi_{_{A_{n,k}}}$$
where $a_{n,k} \in \RR$, $A_{n,k} \in \mathcal{B}$ and, by assumption (2) of Subsection~\ref{sse:U-integrable},
$$\ulim_n\,\Big(\sum_{k=1}^{p_n} a_{n,k} \,\int_{X}\,\chi_{_{A_{n,k}}}\,\,d\mu\Big) < + \infty.$$
The corresponding ultralimit $\widehat f:\widehat X \to \RR$ is measurable (cf. Lemma~\ref{measurability_ulim_fn})
and bounded, hence $\widehat {\mu}_{_\mathcal{U}}-$integrable. Using \eqref{eq:caracteristic} and the linearity of both the $\ulim$ and the integral operator, we obtain
\begin{equation} \label{ulim_sim_func_2}
\int_{\widehat X} \,\ulim_n\,\Big(\sum_{k=1}^{p_n} a_{n,k}\,[\chi_{_{A_{n,k}}}]\Big)\,\,d\,\widehat {\mu}_{{_\mathcal{U}}} \ = \ \ulim_n\,\Big(\sum_{k=1}^{p_n} a_{n,k} \,\int_{X}\,\chi_{_{A_{n,k}}}\,\,d\mu\Big).
\end{equation}

\subsubsection*{\emph{\textbf{Bounded sequences of measurable maps}}}\label{sse:bounded}
Take now a sequence of bounded measurable and $\UU-$integrable functions $f_n: X\to \RR$, and consider the corresponding ultralimit $\widehat f:\widehat X \to \RR$, which is $\widehat {\mu}_{_\mathcal{U}}-$integrable because it is measurable  %(cf. Lemma~\ref{measurability_ulim_fn})
and bounded.

\begin{lemma}\label{simple_func_for_hat_mu_integral}
Given a bounded measurable map $\mathfrak{g}: \widehat X \to \RR$ and $\varepsilon>0$, there exists a simple function $\mathfrak{s}^\varepsilon:\widehat X \to \RR$ supported on UP-sets such that
$$|\mathfrak{g}(z)-\mathfrak{s}^\varepsilon(z)|< \varepsilon \quad \quad \text{at} \quad \widehat {\mu}_{_\mathcal{U}}-\text{almost every}\quad z \,\in\,  \widehat X.$$
\end{lemma}

\begin{proof}
Fix $\mathfrak{g}$ and $\varepsilon >0$. There is a simple map $\mathfrak{r}^\varepsilon = \sum_{k=1}^p \, a_k\,\chi_{_{\Gamma_k}}$, where $a_k \in \RR$ and $\Gamma_k \in L(\widehat{\mathcal{B}})$, such that
$$|\mathfrak{g}(z)-\mathfrak{r}^\varepsilon(z)|< \varepsilon \quad \quad \text{at}\quad \widehat {\mu}_{_\mathcal{U}}-\text{almost every}\quad z \,\in\,  \widehat X.$$
By Proposition~\ref{LBapproximation}, for each $k \in \{1, 2, \cdots, p\}$, we can find $\Upsilon_k \in \widehat{\mathcal{B}}$ such that $\widehat \mu_{{_\mathcal{U}}}(\Gamma_k \,\Delta \,\Upsilon_k)=0$. Rewriting, if needed, the sum that defines $\mathfrak{r}^\varepsilon$, we may assume that $\widehat X$ is the disjoint union of the sets $\Gamma_k$, so $\widehat \mu_{{_\mathcal{U}}}\left(\widehat X \setminus \bigcup_{i=1}^p \,\Gamma_k\right)=0$, and that $\widehat \mu_{{_\mathcal{U}}}\left(\Gamma_i \cap \Gamma_j\right)=0$ for $i\not=j$. Define
$$\Gamma_{p + 1} =X \quad \quad \text{and} \quad \quad \Omega_1 \ = \ \Gamma_1$$
and, recursively, for $1 \leq k \leq p$, set
$$\Omega_{k+1}=\Gamma_{k+1} \setminus \bigcup_{i=1}^k \, \Omega_k.$$
Then, $\widehat \mu_{{_\mathcal{U}}}\left(\Omega_{p+1}\right)=0$ and $\mathfrak{s}^\varepsilon = \sum_{k=1}^p \,a_k\,\chi_{_{\Omega_k}}$ is a simple function supported on UP-sets such that $\mathfrak{r}^\varepsilon(z)=\mathfrak{s}^\varepsilon(z)$ for $\widehat {\mu}_{_\mathcal{U}}-$almost every $z \in \widehat X$ and $\int_{\widehat X} \,\mathfrak{r}^\varepsilon\; d\,\widehat {\mu}_{_\mathcal{U}} \ = \ \int_{\widehat X} \, \mathfrak{s}^\varepsilon \, d\,\widehat {\mu}_{_\mathcal{U}}$.
\end{proof}

For a simple function $\mathfrak{s}:\widehat X \to \RR$ given by $\mathfrak{s}=\sum_{k=1}^p \,a_k\,\chi_{_{\Lambda_k}}$, where $a_k \in \RR$ for all $k$, $\Lambda_k = \Big(\prod_{n \, \in \,\mathbb{N}}\,A_{k,n}\Big)_{\diagup \backsim}$ and $A_{k,n} \in \mathcal{B}$ for every $n$, we can consider the induced sequence of simple functions $(s_n: X \to \RR)_{n \in \NN}$ by setting
\begin{equation}\label{def:sn}
s_n=\sum_{k=1}^p \,a_k\,\chi_{_{A_{k,n}}}.
\end{equation}
Then, from \ref{ulimchi_n} we obtain, using the linearity properties of the $\ulim$,
\begin{equation}\label{ulim_sim_func_1}
\mathfrak{s} \ = \ \ulim_n \,s_n
\end{equation}
and, applying \eqref{eq:caracteristic}, we get
\begin{equation} \label{ulim_sim_func_2}
\int_{\widehat X} \,\mathfrak{s} \,\,d\,\widehat {\mu}_{_\mathcal{U}} \ = \  \ulim_n\,\int_{X}\,s_n\,d\mu.
\end{equation}

Take $\varepsilon >0$ arbitrary. By Lemma~\ref{simple_func_for_hat_mu_integral}, we may find a simple function $\mathfrak{s}^\varepsilon =\sum_{k=1}^p \,a_k\,\chi_{_{\Lambda_k}}$, where $a_k \in \RR$ and $\Lambda_k=\Big(\prod_{\ell \, \in \,\mathbb{N}}\,A_{k,\ell}\Big)_{\diagup \backsim}$ is a UP-set, such that
\begin{equation}\label{eq:aprox}
\left|\widehat f([x_n])- \mathfrak{s}^\varepsilon([x_n])\right| < \frac{\epsilon}{2} \quad \quad \text{at} \quad \widehat {\mu}_{_\mathcal{U}}-\text{almost every}\quad [x_n] \,\in\,  \widehat X
\end{equation}
and
\begin{equation}\label{def_sepsilon}
\left|\int_{\widehat X}\,\widehat f\:d\,\widehat {\mu}_{_\mathcal{U}} -\int_{\widehat X}\,\mathfrak{s}^\varepsilon\:d\,\widehat {\mu}_{_\mathcal{U}} \right| < \frac{\varepsilon}{2}.
\end{equation}

Let $\Big(s^\varepsilon_n\Big)_{n \, \in \, \mathbb{N}}$ be the sequence of simple functions induced by $\mathfrak{s}^\varepsilon$ as explained in \eqref{def:sn}, obeying to \ref{ulim_sim_func_1} and \ref{ulim_sim_func_2}.

\begin{lemma}\label{f_n_s_n_relat}
$\quad \ulim_n \,\int_X\,\left|f_n - s^\varepsilon_n\right|\:d\mu  \ < \ \frac{\epsilon}{2}$.
\end{lemma}

\begin{proof}
Suppose this is not true. Then the set
$$E_n\ = \ \left\{x \in X \colon \left|f_n(x)- s^\varepsilon_n(x)\right| \geq \frac{\varepsilon}{2} \right\}$$
satisfies
$$\ulim_n\,\mu(E_n)>0.$$
Therefore, $\widehat {\mu}_{_\mathcal{U}}(E)>0$ where $E$ is the UP-set $\Big(\prod_{n \, \in \, \mathbb{N}}\,E_n]\Big)_{\diagup \backsim}$. But, from the definiton of $E_n$, we deduce that, for every $[x_n] \in E$,
$$\left|\widehat f([x_n])- \mathfrak{s}^\varepsilon([x_n])\right| \geq \frac{\varepsilon}{2}$$
which contradicts \eqref{eq:aprox}.
\end{proof}

As an immediate consequence of Lemma~\ref{f_n_s_n_relat}, we have
\begin{equation}\label{approxfn_to_sn}
\ulim_n\,\left|\int_X f_n\:d\mu - \int_X \mathfrak{s}^\varepsilon_n\:d\mu\right| = \ulim_n \,\left|\int_X (f_n - \mathfrak{s}^\varepsilon_n)\:d\mu\right| \leq \ulim_n \,\int_X \left|f_n - \mathfrak{s}^\varepsilon_n\right|\:d\mu < \frac{\varepsilon}{2} \bigskip
\end{equation}
On the other hand, by the continuity of the absolute value map,
\begin{eqnarray*}
\ulim_n\,\left|\int_X f_n\:d\mu-\int_X \mathfrak{s}^\varepsilon_n\:d\mu\right| & = & \left|\ulim_n\,\int_X f_n\:d\mu-\ulim_n\,\int_X \mathfrak{s}^\varepsilon_n\:d\mu\right| \nonumber \\
&=& \left|\ulim_n\,\int_X f_n\:d\mu-\int_{\widehat X} \mathfrak{s}^\varepsilon\:d\,\widehat {\mu}_{_\mathcal{U}}\right|.
\end{eqnarray*}
Applying \ref{def_sepsilon} and \ref{approxfn_to_sn}, we finally get
\begin{equation}
\left|\ulim_n\,\int_X f_n\:d\mu-\int_{\widehat X} \widehat f\:d\,\widehat {\mu}_{_\mathcal{U}}\right| \ < \ \frac{\varepsilon}{2}+\frac{\varepsilon}{2} \ = \ \varepsilon.
\end{equation}
Finally, as $\varepsilon>0$ may be chosen arbitrarily small, we conclude that
$$\ulim_n\,\int_X f_n\:d\mu \ = \ \int_{\widehat X} \widehat f\:d\,\widehat {\mu}_{_\mathcal{U}}.$$
\end{proof}

We proceed proving the assertion of Theorem~\ref{teo:main} about dominated sequences.

\subsection{Dominated sequences of measurable maps}

The natural decomposition $f_n=f_n^+ - f_n^-$, with $f_n^+ = \max\,\{f,0\}$ and $f_n^- = \max\,\{-f,0\}$, reduces the proof to the case of a $\UU-$integrable sequence of non-negative $\mu$-integrable functions $f_n: X \to \RR$ dominated by a $\mu-$integrable function $g: X\to \RR$.

Given $M \in \NN$, let $f_n \wedge M: X \to \RR$ be the function
$$f_n \wedge M(x)= \left\{\begin{array}{ll}
f_n(x) & \text{if } \,\,f_n(x)\leq M \\
\\
M & \text{otherwise}.
\end{array}
\right.
$$
In an analogous way, define $g \wedge M$. Notice that, for every $M \in \NN$, both maps $f_n \wedge M$ and $g \wedge M$ are measurable and $\mu-$integrable. Besides, as we assume that $\widehat f$ is defined, we have
$$\widehat{f \wedge M} = \ulim_n \,(f_n \wedge M) = \widehat f \wedge M.$$

We already know that the ultralimit $\widehat f: \widehat X \to \RR$ of the sequence $(f_n)_{n \, \in \, \NN}$ is measurable (cf. Lemma~\ref{measurability_ulim_fn}). Moreover, the sequence $(f_n \wedge M)_{n \, \in \, \NN}$ is $\UU-$integrable since
\begin{eqnarray*}
\mathcal{U} \text{-} \lim_n \,\int_X\,|f_n \wedge M|\,d\mu &\leq& M \\
\mathcal{U} \text{-} \lim_n \,\mu(B_n)=0 \quad &\Rightarrow& \quad \mathcal{U} \text{-} \lim_n \, \int_{B_n}\,|f_n \wedge M|\,d\mu \,\,\leq\,\,  \mathcal{U} \text{-} \lim_n \,M\mu(B_n)=0.
\end{eqnarray*}
Therefore, by Proposition~\ref{prop:auxiliary},
\begin{equation}\label{dominated_result_truncated}
\ulim_n\,\int_X (f_n \wedge M)\:d\mu \ = \ \int_{\widehat X} (\widehat f  \wedge M)\:d\,\widehat {\mu}_{_\mathcal{U}} \quad \quad  \forall \,\,M \in \NN.
\end{equation}
Besides, by the Monotone Convergence Theorem,
$$\int_{\widehat X} \widehat f\:d\,\widehat {\mu}_{_\mathcal{U}}  =  \lim_{M \to +\infty}\,\int_{\widehat X} (\widehat f  \wedge M)\:d\,\widehat {\mu}_{_\mathcal{U}}$$
thus, taking the limit as $M$ goes to $+\infty$ at both hand sides of \ref{dominated_result_truncated}, we obtain
\begin{equation*}
\lim_{M \to +\infty} \,\ulim_n\,\int_X (f_n \wedge M)\:d\mu \ = \ \lim_{M \to +\infty}\, \int_{\widehat X} (\widehat f  \wedge M)\:d\,\widehat {\mu}_{_\mathcal{U}} = \int_{\widehat X} \widehat f\:d\,\widehat {\mu}_{_\mathcal{U}}.
\end{equation*}
To finish the proof of Theorem~\ref{teo:main} we are left to show the following property.

\begin{lemma} Under the assumptions of Theorem~\ref{teo:main},
\begin{equation*}
\lim_{M \to +\infty} \,\ulim_n\,\int_X (f_n \wedge M)\:d\mu \ = \ \ulim_n\,\int_X f_n\:d\mu.
\end{equation*}
\end{lemma}

\begin{proof} Consider, for each $M \in \NN$, the sets
$$A_{n,M}=f_n^{-1}([M,\infty[) \quad \quad \text{and} \quad \quad B_{M}=g^{-1}([M,\infty[).$$
Notice that,
as each $f_n$ is $\mu-$integrable, we have for every $n \in \NN$
$$\lim_{M \to +\infty} \mu(A_{n,M})=0= \lim_{M \to +\infty} \,\int_{A_{n,M}}f_n\;d\mu$$
and, again by the Monotone Convergence Theorem,
$$\lim_{M \to +\infty}\,\int_X (f_n \wedge M)\:d\mu =\int_X f_n\:d\,\mu.$$
Observe now that, for every $n, \,M \in \NN$,
\begin{equation}\label{eq:fn wedge M}
\int_X (f_n \wedge M)\:d\mu \ = \ \int_X f_n\:d\mu -\int_{A_{n,M}} f_n\;d\mu + M\mu(A_{n,M})
\end{equation}
and so, taking into account that
$f_n \geq 0$ and $A_{n,M} \subset B_M$, we conclude that
$$\int_X (f_n \wedge M)\:d\mu \leq  \int_X f_n\:d\mu + M\mu(A_{n,M}) \leq \int_X f_n\:d\mu + M\mu(B_M).$$
Consequently, since $\lim_{M \to +\infty} \, M\mu(B_M)=0$ due to the $\mu-$integrability of $g$, we obtain
\begin{eqnarray}\label{eq:metade da prova}
\lim_{M \to +\infty} \,\ulim_n\,\int_X (f_n \wedge M)\:d\mu & \leq &  \lim_{M \to +\infty} \,\ulim_n\, \Big(\int_X f_n\:d\mu + M\mu(B_M)\Big) \nonumber\\
&=& \ulim_n\, \int_X f_n\:d\mu + \lim_{M \to +\infty} \,M\mu(B_M) \nonumber \\
&=& \ulim_n\, \int_X f_n\:d\mu.
\end{eqnarray}
Conversely, from \eqref{eq:fn wedge M} we also deduce that
$$\int_X (f_n \wedge M)\:d\mu \ \geq \ \int_X f_n\:d\mu -\int_{A_{n,M}} f_n\;d\mu$$
thus, as $0 \leq f_n \leq g$ and $A_{n,M} \subset B_M$,
\begin{eqnarray*}
\int_X (f_n \wedge M)\:d\mu  &\geq&  \int_X f_n\:d\mu -\int_{A_{n,M}} f_n\;d\mu \nonumber \\
&\geq& \int_X f_n\:d\mu -\int_{B_M} f_n\;d\mu \\
&\geq& \int_X f_n\:d\mu -\int_{B_M} g\;d\mu.
\end{eqnarray*}
Therefore,
\begin{eqnarray}\label{eq:a outra metade da prova}
\lim_{M \to +\infty} \,\ulim_n\,\int_X (f_n \wedge M)\:d\mu  &\geq&   \lim_{M \to +\infty} \,\ulim_n\, \Big(\int_X f_n\:d\mu -\int_{B_M} g\;d\mu\Big) \nonumber \\
&=& \ulim_n\, \int_X f_n\:d\mu.
\end{eqnarray}
The proof ends by considering both \eqref{eq:metade da prova} and \eqref{eq:a outra metade da prova}.
\end{proof}

\begin{remark}
\emph{The key ingredient of the previous argument is the uniformity in the variable $n$ of the $\UU-$limit of both sequences $(\mu(A_{n,M}))_{n,\,M \, \in \, \NN}$ and $(\int_{A_{n,M}} f_n\;d\mu)_{n,\,M \, \in \, \NN}$ as $M$ goes to $+\infty$. Thus, Theorem~\ref{teo:main} is still valid if, instead of domination, we assume any other property of the sequence $(f_n)_{n \, \in \NN}$ which ensures that the previous $\UU-$limits are uniform in $n$. See Section~\ref{se:proof-of-Theorem-B}.}
\end{remark}

\section{Proof of Theorem~\ref{teo:second-main}}\label{se:proof-of-Theorem-B}

As in the previous section, we will assume that $f_n \geq 0$ for every $n \in \NN$. To prove Theorem~\ref{teo:second-main} we just need to conclude that the $\UU-$integrability (cf. Subsection~\ref{sse:U-integrable}) of $(f_n)_{n \, \in \NN}$ implies that the $\UU-$limit, as $M$ goes to $+\infty$, of both sequences $(\mu(A_{n,M}))_{\{n,\,M \, \in \, \NN\}}$ and $(\int_{A_{n,M}} f_n\;d\mu)_{\{n,\,M \, \in \, \NN\}}$ is $0$ uniformly in $n$.

Concerning the first sequence, we notice that the assumption (2) of the definition of $\UU-$integrability guarantees that there exists $K>0$ such that
$$\Big\{n \in \NN \colon \,\int_X \,f_n < K\Big\} \quad \in \,\,\mathcal{U}.$$
Therefore, for the elements of this set we have
$$0 \leq \mu(A_{n,M}) \leq \frac{\int_X \,f_n \, d\mu}{M} < \frac{K}{M}$$
and so $\ulim_M \, \mu(A_{n,M}) = 0$ uniformly in $n$.

\medskip

From the $\mu-$integrability assumption of each $f_n$ we know that $\lim_{M \, \to \, +\infty} \,\int_{A_{n,M}} f_n\;d\mu=0$ for every $n \in \NN$. Given $n \in \NN$ and $\varepsilon > 0$, consider the set
$$W_n \,:=\,\Big\{M \in \NN \colon \, \int_{A_{n,M}} f_n \; d\mu < \varepsilon\Big\}.$$
It turns out that each $W_n$ is a co-finite set since $A_{n,M+1} \subset A_{n,M}$ and so, as $f_n \geq 0$, if $M \in W_n$ then $M+1 \in W_n$. This means that $W_n = [\gamma_n,+\infty[$ for some $\gamma_n \in \NN$ which is the least element of $W_n$. We note that one cannot have $\ulim_n \,\gamma_n=+\infty$, otherwise $\ulim_n \,(\gamma_n-1) = +\infty$ although, by the minimality of $\gamma_n$, we must also have
$$\Big\{n \in \NN \colon \, \gamma_n-1 \in W_n \Big\} \quad  \notin \,\,\UU.$$
Yet, these two properties contradict the condition (3) of the definition of $\UU-$integrability of $(f_n)_{n}$. Indeed, for every sequence $(L_n)_{n \, \in \, \NN}$ of elements of $\NN$ satisfying $\lim_{n \, \to \, +\infty}\, L_n = + \infty$, we have $\lim_{n \, \to \, +\infty}\, \mu(A_{n,L_n}) =0$, which, due to condition (3), yields $\ulim_n \int_{A_{n,L_n}} f_n \; d\mu = 0$.
Therefore, there exists $M_0 \in \NN$ such that $\ulim \gamma_n = M_0$. Consequently,
$$\Big\{n \in \NN \colon \, M \in W_n \quad \forall \, M \geq M_0 \Big\} \quad \in \,\ \UU.$$
This is the uniformity we were looking for.

\section{Changing the ultrafilter}\label{se:proof-of-Corollary-B}

Consider a compact metric space $X$, a Borel probability measure $\mu$ in $X$, a measurable map $T: X \to X$, a non-principal ultrafilter $\mathcal{U}$ in $\NN$ and a measurable bounded function $\varphi: X \to \RR$. We will show that there is a Borel shift-invariant probability measure $\eta_{_\mathcal{U}}$ in $\beta \NN$ such that, for every $x \in X$, the space mean in $\beta \NN$ given by
$$\int_{\beta \NN}\,\plim_n \, \Big(\frac{\varphi(x) + \varphi(T(x)) + \cdots + \varphi(T^{n-1}(x))}{n}\Big) \,d\eta_{_\mathcal{U}}(p)$$
is well defined.

A natural dynamics in $\beta \NN$ is determined by the extension to the Stone-$\check{C}$ech compactification $\beta \NN$ of the map $s\colon \mathbb{N} \to \mathbb{N}$ such that $s(n)=n+1$. More precisely, we have defined a continuous map (whom we also call shift)
\begin{eqnarray*}
S \colon \beta \NN \quad &\to& \quad \beta \NN \\
p \quad &\mapsto& \quad S(p) = \{A \colon A - 1 \in p\}
\end{eqnarray*}
where $A-1=\{n \in \NN \colon n+1 \in A\}.$ It is easy to verify that $S$ commutes with the map $\sigma \colon \ell_\infty(\mathbb{R}) \to \ell_\infty(\mathbb{R})$ given by $\sigma\left((b_n)_{n \, \in \, \mathbb{N}}\right) = (b_{n+1})_{n \, \in \, \mathbb{N}}$; that is, given a bounded sequence $(b_n)_{n \, \in \, \NN}$, we have $\splim_n \,\,b_n = \plim_n \,\,b_{n+1}.$ Therefore, by the linearity of the ultralimits we have, for every $k \in \NN$,
\begin{equation}\label{eq:ultra-average}
\frac{1}{k} \,\,\Big(p + S(p) + \cdots + S^{k-1}(p)\Big) - \lim_n \,\,b_n = \plim_n \,\,\Big(\frac{b_n + b_{n+1}+ \cdots + b_{n+k-1}}{k}\Big).
\end{equation}

The space $C^0(\beta \NN)$ of all continuous maps $f \colon \beta \NN \to \mathbb{R}$ (which have compact support) is isometrically isomorphic to the space $\ell_\infty(\mathbb{R})$ of bounded sequences of $\mathbb{R}$. Indeed, if $\tau \colon \mathbb{N} \to \beta \NN$ is the inclusion that takes $n_0 \in \mathbb{N}$ to the principal ultrafilter $\mathcal{U}_{n_0}$, then the maps
\begin{eqnarray*}
G: \ell_\infty(\mathbb{R})\, \to\, C^0(\beta \NN) \quad &\rightarrowtail & \quad G\Big(\overline{b}:=(b_n)_{n \, \in \, \mathbb{N}}\Big) = \psi_{\overline{b}}, \,\,\,\text{ where } \,\,\,\psi_{\overline{b}}\,(p) = \plim_n\,b_n  \\
H: C^0(\beta \NN)\, \to \,\ell_\infty(\mathbb{R}) \quad &\rightarrowtail & \quad H(\psi) = \psi \circ \tau
\end{eqnarray*}
are linear norm-preserving isomorphisms between the two spaces.

Having fixed the ultrafilter $\mathcal{U} \in \beta \NN$, consider the operator $\mathcal{L} \colon  C^0(\beta \NN) \to  \RR$ which assigns to each $\psi \in C^0(\beta \NN)$ (which, after the identification of $C^0(\beta \NN)$ with $\ell_\infty(\mathbb{R})$, may be seen as a bounded sequence $(a_n)_{n\, \in \,\mathbb{N}}$) the real number
$$\mathcal{L}(\psi) = \ulim_n \,\,\Big(\frac{a_1 + a_2 + \cdots + a_{n}}{n}\Big).$$
The operator $\mathcal{L}$ is linear, positive, and $\mathcal{L}(\textbf{1})=1$. Moreover, given $\psi \in C^0(\beta \NN)$,
the map $\psi \circ S$ is represented by the bounded sequence $(a_{n+1})_{n\, \in\, \mathbb{N}}$, and so
\begin{eqnarray*}
\mathcal{L}(\psi \circ S) &=& \ulim_n \,\,\Big(\frac{a_2 + a_3 + \cdots + a_{n+1}}{n}\Big) \\
&=& \ulim_n \,\,\Big(\frac{a_1 + a_2 + \cdots + a_{n}}{n} - \frac{a_1}{n} + \frac{a_{n+1}}{n}\Big) \\
&=& \ulim_n \,\,\Big(\frac{a_1 + a_2 + \cdots + a_{n}}{n} \Big)  = \mathcal{L}(\psi).
\end{eqnarray*}
Therefore, by the Representation Theorem of Riesz-Markov-Kakutani there is a unique regular Borel probability measure $\eta_{_\mathcal{U}}$ on $\beta \NN$ such that
$$\mathcal{L}(\psi) = \int_{\beta \NN}\,\psi\,d\eta_{_\mathcal{U}} \quad \quad \forall \,\,\psi \,\in \,C^0(\beta \NN).$$
For instance, if we take $x_0 \in X$, a bounded map $\varphi \colon X \to \RR$ and the bounded sequence
$$(a_n)_{n \, \in \, \NN}\,\,:=\,\, \Big(\varphi(T^{n}(x_0))\Big)_{n \, \in \, \NN}$$
then we conclude that
$$\ulim_n \,\,\Big(\frac{\varphi(x_0) + \varphi(T(x_0)) + \cdots + \varphi(T^{n-1}(x_0))}{n}\Big) = \int_{\beta \NN}\,\plim_n \,\varphi(T^n(x_0)) \,d\eta_{_\mathcal{U}}(p).$$
We recall that the map $p \in \beta \NN \mapsto \plim_n \,\varphi(T^n(x_0))$ is the continuous Stone-$\check{C}$ech extension of the continuous map $n \in \NN \mapsto \varphi(T^n(x_0))$, and so it is $\eta_{_\mathcal{U}}-$integrable.

An important consequence of the way the probability measure $\eta_{_\mathcal{U}}$ was obtained is the fact that $\eta_{_\mathcal{U}}$ is $S-$invariant. Indeed, in normal Hausdorff spaces (such as the compact Hausdorff $\beta \NN$) the characterization of invariant probabilities made be done using continuous maps (cf. \cite[Theorem 6.2]{Walters}); and, for every $\psi \in C^0(\beta \NN)$,

\begin{equation*}
\int_{\beta \NN}\,\psi \circ S\,d\eta_{_\mathcal{U}} = \mathcal{L}(\psi \circ S) = \mathcal{L}(\psi) = \int_{\beta \NN}\,\psi\,d\eta_{_\mathcal{U}}.
\end{equation*}
Consequently, by the Ergodic Theorem of Birkhoff, given an $L(\widehat{\mathcal{B}})-$measurable and $\eta_{_\mathcal{U}}-$integrable map $\psi: \beta \NN \to \RR$, the sequence of averages

$$\Big(\frac{\psi(p) + \psi(S(p)) + \psi(S^2(p)) + \cdots + \psi(S^{n-1}(p))}{n}\Big)_{n \, \in \, \mathbb{N}}$$
converges at $\eta_{_\mathcal{U}}$ almost every $p \in \beta \NN$, thus defining an $L(\widehat{\mathcal{B}})-$measurable and $\eta_{_\mathcal{U}}-$integrable map $\widetilde{\psi}$ such that $\widetilde{\psi}\circ S=\widetilde{\psi}$. Moreover, $\int_{\beta \NN}\,\widetilde{\psi} \,d\eta_{_\mathcal{U}} = \int_{\beta \NN}\,\psi \,d\eta_{_\mathcal{U}}$, that is,

\begin{equation*}
\int_{\beta \NN}\,\lim_n \frac{\psi(p) + \psi(S(p)) + \cdots + \psi(S^{n-1}(p))}{n} \,d\eta_{_\mathcal{U}}(p) = \int_{\beta \NN}\,\psi \,d\eta_{_\mathcal{U}}.
\end{equation*}
Considering the unique bounded sequence $(a_m)_{m \in \mathbb{N}}$ which represents $\psi$ and the equation \eqref{eq:ultra-average}, the previous equality may be rewritten as

\begin{equation*}
\int_{\beta \NN}\,\lim_n \frac{\Big(\plim_m \, a_m\Big) + \cdots + \Big(\plim_m \,a_{m + n-1}\Big)}{n} \,d\eta_{_\mathcal{U}}(p) = \int_{\beta \NN}\,\plim_m \, a_m \,d\eta_{_\mathcal{U}}(p).
\end{equation*}
In particular, if $x_0 \in X$, $\varphi \colon X \to \RR$ is a bounded map and we consider the bounded sequence

$$(a_m)_{m \, \in \, \NN}\,:=\,\Big(\frac{\varphi(x_0) + \varphi(T(x_0)) + \cdots + \varphi(T^{m-1}(x_0))}{m}\Big)_{m \, \in \, \NN}\,:=\,\psi_{\{x_0,\, \varphi,\, T\}} \,\,\in \beta \NN$$
then we deduce that the space mean in $\beta \NN$ of the ultralimits of the Birkhoff averages along the orbit of $x_0$ by $T$, namely

$$\int_{\beta \NN}\,\plim_m \, \Big(\frac{\varphi(x_0) + \varphi(T(x_0)) + \cdots + \varphi(T^{m-1}(x_0))}{m}\Big) \,d\eta_{_\mathcal{U}}(p)$$
is well defined and is given by $\int_{\beta \NN}\,\widetilde{\psi}_{\{x_0,\, \varphi,\, T\}}(p) \,d\eta_{_\mathcal{U}}(p)$, where $\widetilde{\psi}_{\{x_0,\, \varphi,\, T\}}$ stands for the Birkhoff limit of the averages of the observable $\psi_{\{x_0,\, \varphi,\, T\}} \in C^0(\beta \NN)$ with respect to the dynamics $S$ and the probability measure $\eta_{_\mathcal{U}}$, that is,

\begin{eqnarray*}
\widetilde{\psi}_{\{x_0,\, \varphi,\, T\}}(p)=\lim_n \,\frac{\Big(\plim_m \, a_m\Big) + \Big(\plim_m \, a_{m+1}\Big) + \cdots + \Big(\plim_m \,a_{m + n-1}\Big)}{n}.
\end{eqnarray*}

\medskip
\noindent at $\eta_{_\mathcal{U}}$ almost every $p \in \beta \NN$.

\section{Example} Take $X=[0,1]$ and consider the dynamics $T(x) = \frac{x}{2}$ if $x \neq 0$, $T(0)=0$, the Dirac measure $\mu=\delta_0$ supported on $\{0\}$ and the map $\varphi = \text{Identity}_{[0,1]}$. Observe that the non-wandering set of $T$ is $\{0\}$ and that $\mathcal{U} \text{-} \lim_n\,\,\,\frac 1 n \,\sum_{j=0}^{n-1}\,T^j(x)=0$ for every $x \in \,[0,1]$ and any ultrafilter $\UU$, since the sequence $(T^n(x))_{n \in \NN}$ converges to $0$. Therefore, by Corollary~\ref{cor:main}, the map $\widehat \varphi_{_\mathcal{U}}$ is $0$ at $\widehat \mu_{_\mathcal{U}}$ almost every point of $\widehat{[0,1]}$, because $\widehat \varphi_{_\mathcal{U}} \geq 0$ and
$$\int_{_{\widehat X}}\,\widehat \varphi_{_\mathcal{U}}\,d\,\widehat \mu_{_\mathcal{U}} \,= \,\ulim_n \,\frac 1 n \,\sum_{j=0}^{n-1}\, \int_{_{X}} \,\varphi\circ T^j \, d\mu \,=\, \mathcal{U} \text{-} \lim_n\,\,\,\frac 1 n \,\sum_{j=0}^{n-1}\,T^j(0) = 0.$$
We note that $\int_X \,\varphi \,\,d\mu = \varphi(0)=0$ as well. If we consider instead the Dirac measure $\mu=\delta_1$ supported on $\{1\}$, then the previous argument also proves that, for every ultrafilter $\UU$,
the map $\widehat \varphi_{_\mathcal{U}}$ is $0$ at $\widehat \mu_{_\mathcal{U}}$ almost every point of $\widehat{X}$.
However, now we get $\int_X \,\varphi \,\,d\mu = \varphi(1)=1$.


\begin{thebibliography}{10}

\bibitem{AB1998}
C.D. Aliprantis, O. Burkinshaw.
\newblock \emph{Principles of Real Analysis.}
\newblock Academic Press, San Diego, 3rd edition, 1998.

\bibitem{Anderson}
R.M. Anderson.
\newblock \emph{Star-finite representations of measure spaces.}
\newblock Trans. Amer. Math. Soc. 271:2 (1982) 667--687.

\bibitem{Bankston}
P. Bankston.
\newblock \emph{Ultraproducts in topology.}
\newblock General Top. Appl. 7 (1977) 283--308.

\bibitem{Bankston2}
P. Bankston.
\newblock \emph{Topological reduced products via good ultrafilters.}
\newblock General Top. Appl. 10:2 (1979) 121--137.

\bibitem{CM2014}
M. Carvalho, F. Moreira.
\newblock \emph{A note on the Ergodic Theorem.}
\newblock Qual. Theory Dyn. Syst. 13:2 (2014) 253--268.

\bibitem{CV-M-Marinacci}
S. Cerreia-Vioglio, F. Maccheroni, M. Marinacci.
\newblock \emph{Ergodic theorems for lower probabilities.}
\newblock Proc. Amer. Math. Soc.  144:8 (2016) 3381--3396.

%\bibitem{Colli-Vargas}
%E. Colli, E. Vargas.
%\newblock \emph{Non-trivial wandering domains and homoclinic bifurcations.}
%\newblock Ergod. Th. \& Dynam. Sys. 21 (2001) 1657--1681.

\bibitem{Cutland}
N. Cutland.
\newblock \emph{Nonstandard measure theory and its applications.}
\newblock Bull. London Math. Soc. 15 (1983) 529--589.

%\bibitem{Feldman}
%M. B. Feldman.
%\newblock \emph{A Proof of Lusin's Theorem.}
%\newblock Amer. Math. Monthly 88 (1981) 191--192.

\bibitem{Goldblatt}
R. Goldblatt.
\newblock \emph{Lectures on the Hyperreals.}
\newblock Springer Verlag, 1998.

%\bibitem{Gaunersdorfer}
%A. Gaunersdorfer.
%\newblock \emph{Time averages for heteroclinic attractors.}
%\newblock SIAM J. Math. Anal. 52 (1992) 1476--1489.

%\bibitem{Kamae}
%T. Kamae.
%\newblock \emph{A simple proof of the ergodic theorem using nonstandard analysis.}
%\newblock Israel J. Math. 42:4 (1982) 284-290.

\bibitem{Katznelson-Weiss}
Y. Katznelson, B. Weiss.
\newblock \emph{A simple proof of some ergodic theorems.}
\newblock Israel J. Math. 42:4 (1982) 291--296.

\bibitem{Kiriki-Soma}
S. Kiriki, T. Soma.
\newblock \emph{Takens' last problem and existence of non-trivial wandering domains.}
\newblock Adv. Math. 306 (2017) 524--588.

\bibitem{Krengel}
U. Krengel.
\newblock \emph{Ergodic Theorems.}
\newblock De Gruyter Studies in Mathematics 6, 1985.

%\bibitem{Li-Wu}
%J.J. Li and M. Wu.
%\newblock \emph{Generic property of irregular sets in systems satisfying the specification property.}
%\newblock Discrete Contin. Dyn. Syst. 34 (2014) 635--645.

%\bibitem{Parthasarathy}
%K. R.~Parthasarathy.
%\newblock \emph{Introduction to Probability and Measure.}
%\newblock MacMillan, 1977.

%\bibitem{Oxtoby-Ulam}
%J.C. Oxtoby, S.M. Ulam.
%\newblock \emph{Measure-preserving homeomorphisms and metrical transitivity.}
%\newblock Ann. of Math. 42:2 (1941) 874--920.

\bibitem{Petersen}
K. Petersen.
\newblock \emph{Ergodic Theory.}
\newblock Cambridge Studies in Advanced Mathematics 2, Cambridge University Press, 1989.

\bibitem{Sierpinski}
W. Sierpinski.
\newblock \emph{Fonctions additives non compl\'{e}tement additives et fonctions non mesurables.}
\newblock Fund. Math. 30:1 (1938) 96--99.

%\bibitem{Takens}
%F. Takens.
%\newblock \emph{Heteroclinic attractors: time averages and moduli of topological stability.}
%\newblock Bol. Soc. Bras. Mat. 25:1 (1994) 107--120.

\bibitem{Takens2}
F. Takens.
\newblock \emph{Orbits with historic behaviour, or non-existence of averages.}
\newblock Nonlinearity 21 (2008) T33--T36.

%\bibitem{Thompson}
%D. Thompson.
%\newblock \emph{The irregular set for maps with the specification property has full topological pressure.}
%\newblock Dyn. Syst. 25:1 (2010) 25--51.

\bibitem{Walters}
P. Walters.
\newblock \emph{An Introduction to Ergodic Theory.}
\newblock Springer-Verlag, 1982.

\bibitem{Willard}
S. Willard.
\newblock \emph{General Topology.}
\newblock Addison-Wesley, 1970.


\end{thebibliography}
\end{document}